\documentclass[a4paper,12pt,leqno]{article}
\usepackage[margin=3cm]{geometry}
\usepackage{amssymb}
\usepackage{amsxtra}
\usepackage{amsthm}
\usepackage[T1]{fontenc}
\usepackage{lmodern}
\usepackage[pagebackref,breaklinks]{hyperref}
\usepackage[lite,abbrev,shortalphabetic]{amsrefs}
\usepackage[shortlabels]{enumitem}
\usepackage{colonequals}
\theoremstyle{definition}
\newtheorem{defn}{Definition}[section]

\newtheorem{remark}[defn]{Remark}
\newtheorem{notation}[defn]{Notation}
\theoremstyle{plain}
\newtheorem{lemma}[defn]{Lemma}
\newtheorem{prop}[defn]{Proposition}
\newtheorem{cor}[defn]{Corollary}
\newtheorem{theorem}[defn]{Theorem}
\numberwithin{equation}{section}
\setlist[enumerate,1]{(a)}

\newcommand{\N}{\mathbb{N}}
\newcommand{\Z}{\mathbb{Z}}
\newcommand{\Q}{\mathbb{Q}}
\newcommand{\R}{\mathbb{R}}

\newcommand{\bc}{\mathbf{c}}
\newcommand{\bs}{\mathbf{s}}
\newcommand{\bu}{\mathbf{u}}
\newcommand{\bv}{\mathbf{v}}
\newcommand{\bw}{\mathbf{w}}

\newcommand{\cC}{\mathcal{C}}
\newcommand{\cD}{\mathcal{D}}
\newcommand{\cP}{\mathcal{P}}
\newcommand{\cS}{\mathcal{S}}

\newcommand{\alt}{\mathrm{alt}}
\newcommand{\lex}{\mathrm{lex}}

\newcommand{\precalt}{\preceq_{\alt}}
\newcommand{\preclex}{\preceq_{\lex}}
\newcommand{\precnalt}{\prec_{\alt}}
\newcommand{\precnlex}{\prec_{\lex}}

\begin{document}
\title{Alternating geometric progressions modulo one and Sturmian words}
\author{Qing Lu\thanks{School of Mathematical Sciences, Beijing Normal University, Beijing
100875, China; email: \texttt{qlu@bnu.edu.cn}.} \and Weizhe 
Zheng\thanks{Morningside Center of Mathematics, Academy of Mathematics and 
Systems Science, Chinese Academy of Sciences, Beijing 100190, China; 
University of Chinese Academy of Sciences, Beijing 100049, China; email: 
\texttt{wzheng@math.ac.cn}.}}
\date{} \maketitle

\begin{abstract}
Let $b\ge 2$ be an integer. Using Sturmian words we describe all irrational 
real numbers $\xi$ such that the image in $\R/\Z$ of the sequence $(\xi 
(-b)^n)_{n\ge 0}$ is contained in an interval of length 
$b^{-1}+b^{-2}-b^{-3}$. In previous work we showed that the image cannot be 
contained in a shorter interval. 
\end{abstract}

\section{Introduction}
Distribution modulo one of geometric progressions has been the subject of 
extensive studies. The famous question considered by Mahler \cite{Mahler} 
concerning the case of common ratio $3/2$ remains open. A theorem of Dubickas 
\cite[Theorem~1]{Dubickas-BLMS} implies that for any integer $b\ge 2$ and any 
irrational real number $\xi$, the numbers $\{\xi b^n\}$, $n\ge 0$ cannot be 
contained in an interval of length less than $1/b$. Here $\{x\}\colonequals 
x-\lfloor x\rfloor $ denotes the fractional part of a real number~$x$. 
Bugeaud and Dubickas \cite[Theorem 2.1]{BD} showed that for an irrational 
real number $\xi$, the numbers $\{\xi b^n\}$, $n\ge 0$ are contained in an 
interval of length equal to $1/b$ if and only if the representation 
$.w_1w_2\dots$ of $\{\xi\}$ in base $b$ is given by a Sturmian word 
$w_1w_2\dots$ over an alphabet consisting of two adjacent integers. As 
pointed out in \cite{AG}, the case $b=2$ was proved by Veerman \cite[Theorem 
2.1]{Veerman} and rediscovered several times. 

For alternating geometric progressions with integral common ratio, we showed 
in previous work \cite[Theorem 1.1]{LZ} that for any irrational real 
number~$\xi$, $(\pi(\xi (-b)^n))_{n\ge 0}$ is not contained in any interval 
of $\R/\Z$ of length less than $b^{-1}+b^{-2}-b^{-3}$. Here $\pi\colon \R\to 
\R/\Z$ denotes the projection. In this paper, we describe the irrational real 
numbers $\xi$ such that $(\pi(\xi (-b)^n))_{n\ge 0}$ is contained in an 
interval of $\R/\Z$ of length equal to $b^{-1}+b^{-2}-b^{-3}$. 

To state our results we need some terminology from combinatorics on words. An 
infinite word $w_1w_2\dots$ is said to be \emph{aperiodic} if there do not 
exist integers $k,N\ge 1$ such that $w_n=w_{n+k}$ for all $n\ge N$. The 
\emph{slope} of an infinite word $w_1w_2\dots$ over the alphabet $\Z$ is 
defined to be $\lim_{n\to \infty}\frac{\sum_{i=1}^n w_i}{n}$. An infinite 
word over $\Z$ with irrational slope is necessarily aperiodic. An infinite 
word $s_1s_2\dots$ over the alphabet $\{0,1\}$ is \emph{Sturmian} if and only 
if there exists an irrational $\theta\in (0,1)$  and a $\rho\in \R$ such that 
\begin{equation}\label{eq:S1}
s_n=\lfloor (n+1)\theta +\rho\rfloor -\lfloor 
n\theta +\rho\rfloor
\end{equation}
for all $n\ge 1$ or 
\begin{equation}\label{eq:S2}
s_n=\lceil (n+1)\theta +\rho\rceil -\lceil 
n\theta +\rho\rceil
\end{equation} 
for all $n\ge 1$. These words have slope $\theta$. The \emph{characteristic 
Sturmian word} $\bc_\theta=c_1c_2\dots$ of slope $\theta$ is defined by 
\[c_n=\lfloor (n+1)\theta\rfloor -\lfloor 
n\theta \rfloor=\lceil (n+1)\theta\rceil -\lceil n\theta \rceil
\] 
for all $n\ge 1$. We refer to \cite[Chapter~2]{Lothaire} for a detailed 
introduction to Sturmian words. 

For a word $\bw=w_1w_2\dots$, we write $D\bw=w_1w_1w_2w_2\dots$ and 
$T\bw=w_2w_3\dots$. Note that $\bw$ has slope $\theta$ if and only if $D\bw$ 
has slope $\theta$. Moreover, $T$ preserves Sturmian words. For a letter~$a$ 
and $l\ge 0$ (finite or infinite), we write $a^l$ for the word $aa\dots $ of 
length $l$. 

\begin{notation}
We let $\cD$ denote the collection of infinite words over the alphabet 
$\{0,1\}$ that are of the form $0^l(D\bs)$ or $1^l(D\bs)$, where $l\ge 0$ and 
$\bs$ is a Sturmian word. We let $\cC$ denote the subset of $\cD$ consisting 
of words $\bw\in \cD$ such that $T^n\bw$ equals $011(D\bc)$ or $100(D\bc)$ 
for some $n\ge 0$ and some characteristic Sturmian word~$\bc$ (which is 
necessarily $\bc_\theta$, where $\theta$ is the slope of $\bw$). 
\end{notation}

For a word $\bw=w_1w_2\dots$ over a finite set of integers and a real number 
$r$ satisfying $\lvert r\rvert<1$, we write $t_r(\bw)=\sum_{i=1}^\infty 
w_ir^i\in \R$. In particular, $t_{-1/b}(\bw)$ is the real number 
$.w_1w_2\dots$ in base $-b$. For comments on the history of negative-base 
number systems, see \cite[Section 4.1]{Knuth}. 

Our main result is the following.

\begin{theorem}\label{t:main}
Let $b\ge 2$ be an integer and let $\xi$ be an irrational real number. Then 
the sequence $(\pi(\xi(-b)^n))_{n\ge 0}$ is contained in an interval of 
$\R/\Z$ of length $b^{-1}+b^{-2}-b^{-3}$ if and only if 
$\xi=g/(b+1)+t_{-1/b}(\bw)$, where $g$ is an integer and $\bw\in \cD$. In 
this case, $\xi$ is transcendental and the endpoints of the interval are 
\begin{equation}\label{eq:ex}
\pi(g/(b+1)+t_{-1/b}(100D\bc_{\theta})),\quad \pi(g/(b+1)+t_{-1/b}(011D\bc_{\theta})),
\end{equation}
where $\theta$ denotes the slope of $\bw$, and the sequence is contained in a 
semiopen interval with these endpoints. In particular, the sequence is 
contained in an open interval of $\R/\Z$ of length $b^{-1}+b^{-2}-b^{-3}$ if 
and only if $\xi=g/(b+1)+t_{-1/b}(\bw)$, where $g$ is an integer and $\bw\in 
\cD\backslash \cC$.  
\end{theorem} 

Since $t_{-1/b}$ restricts to an injection $\cD\to \R\backslash \Q$, Theorem 
\ref{t:main} implies that the set of open intervals $I\subseteq \R/\Z$ of 
length $b^{-1}+b^{-2}-b^{-3}$ such that there exists an irrational real 
number $\xi$ with $(\pi(\xi(-b)^n))_{n\ge 0}$ contained in $I$ has the 
cardinality of the continuum. Indeed, for any irrational $\theta\in (0,1)$, 
$D\bc_\theta\in \cD\backslash \cC$ and consequently 
$\xi=t_{-1/b}(D\bc_\theta)$ satisfies this property with 
$I=\pi((t_{-1/b}(100D\bc_{\theta}), t_{-1/b}(011D\bc_{\theta})))$. 

The aforementioned result of Veerman \cite[Theorem 2.1]{Veerman} implies the 
following characterization of Sturmian words: An aperiodic word $\bw$ over 
$\{0,1\}$ is Sturmian if and only if there exists an infinite word $\bu$ over 
$\{0,1\}$ such that $0\bu\preclex T^n\bw \preclex 1\bu$ for all $n\ge 0$. 
Here $\preclex$ denotes the lexicographical order. We will briefly review 
this characterization in Section~\ref{s:pos}. For a nice survey of this 
subject, see \cite{AG}. 

In negative-base number systems, the order on real numbers is compatible with 
the \emph{alternate order} $\precalt$ on words rather than the 
lexicographical order $\preclex$. For infinite words $\bv=v_1v_2\dots$ and 
$\bw=w_1w_2\dots$ over $\Z$, we write $\bv \precalt \bw$ if $N\bv\preclex 
N\bw$, where $N\bv=(-v_1)v_2(-v_3)v_4\dots$ and similarly for $N\bw$. It is 
easy to see that, for infinite words $\bv$ and $\bw$ over $\{0,1\}$ and 
$0<r<1/2$, $\bv\precalt \bw$ if and only if $t_{-r}(\bv)\le t_{-r}(\bw)$ 
(Lemma \ref{l:t-r}). Theorem \ref{t:main} implies the following 
characterization of the class $\cD$. 

\begin{cor}\label{c:alt}
An aperiodic word $\bw$ over $\{0,1\}$ belongs to $\cD$ if and only if there 
exists an infinite word $\bu$ over $\{0,1\}$ such that $100\bu\precalt 
T^n\bw\precalt011\bu$ for all $n\ge 0$. In this case, $\bu=D\bc_\theta$ and 
\[100\bu=\inf\nolimits^\alt_{n\ge 0} T^n\bw, \quad 011\bu=\sup\nolimits^\alt_{n\ge 0} T^n \bw.\] 
Here $\theta$ denotes the slope of $\bw$, and $\inf^\alt$ and $\sup^\alt$ 
denote respectively the infimum and supremum with respect to the order 
$\precalt$. 
\end{cor}

More generally, we prove the following result about $t_{-r}$.

\begin{theorem}\label{t:r}
Let $\bv=v_1v_2\dots$ be an aperiodic word over a finite set of integers and 
let $r$ be a real number satisfying $0<r<1$. Then the sequence 
$(t_{-r}(T^n\bv))_{n\ge 0}$ is contained in an interval of length $r+r^2-r^3$ 
if and only if there exist $g\in \Z$ and $\bw=w_1w_2\dotso\in \cD$ such that 
$v_i=-g+w_i$ for all $i\ge 1$. In this case, the endpoints of the interval 
are 
\[gr/(1+r)+t_{-r}(100D\bc_{\theta}),\quad gr/(1+r)+t_{-r}(011D\bc_{\theta}),\]
where $\theta$ denotes the slope of $\bw$, and the sequence is contained in a 
semiopen interval with these endpoints.  Furthermore, in this case, the 
sequence is contained in the open interval with these endpoints if and only 
if $\bw\notin \cC$. 
\end{theorem}

For alternating geometric progressions with rational common ratio $-p/q$, 
where $p>q>1$ are integers, we showed in \cite[Theorem 1.1]{LZ} that for any 
real number $\xi\neq 0$, the sequence $(\pi(\xi (-p/q)^n))_{n\ge 0}$ is not 
contained in any interval of $\R/\Z$ of length less than $(r+r^2-r^3)/q$, 
where $r=q/p$. In subsequent work, we will use Theorem \ref{t:r} to study the 
question whether the sequence can lie in an interval of length equal to 
$(r+r^2-r^3)/q$. The analogous question for powers of positive rational 
numbers was studied by several authors (\cite{FLP}, \cite{Bugeaud-linear}, 
\cite{Dubickas-small}, \cite{Dubickas-large}). 

This paper is organized as follows. In Section~\ref{s:pre}, we present some 
preliminary results. In Section~\ref{s:pos}, we state and prove an analogue 
of Theorem \ref{t:r} for $t_r$ and deduce the theorems of Veerman and 
Bugeaud--Dubickas. This serves as a warm-up for the more involved proofs, 
given in Section~\ref{s:proof}, of the results stated in the introduction. 

\subsection*{Acknowledgments}
This work was partially supported by the National Natural Science Foundation 
of China (grant numbers 12125107, 12271037, 12288201) and the Chinese Academy 
of Sciences Project for Young Scientists in Basic Research (grant number 
YSBR-033). 

\section{Preliminaries}\label{s:pre}

The following characterization of Sturmian words as balanced aperiodic words 
was proved by Morse and Hedlund (see \cite{MH2} or \cite[Theorem 
2.1.13]{Lothaire}). 

\begin{lemma}\label{l:Sturm}
Let $\bs$ be an aperiodic word over the alphabet $\{0,1\}$. Then $\bs$ is 
Sturmian if and only if there does not exist any finite word $\bu$ such that 
$0\bu 0$ and $1\bu 1$ are subwords of $\bs$.  
\end{lemma}

The ``only if'' part of Lemma \ref{l:Sturm} can be strengthened as follows. 

\begin{lemma}\label{l:Sturm2}
Let $\bs$ and $\bs'$ be Sturmian words over $\{0,1\}$ of the same slope. Then 
there does not exist any finite word $\bu$ such that $0\bu 0$ is a subword of 
$\bs$ and $1\bu 1$ is a subword of $\bs'$. 
\end{lemma}

\begin{proof}
This follows immediately from the ``only if'' part of Lemma \ref{l:Sturm} and 
the fact that $\bs$ and $\bs'$ have the same set of finite subwords 
\cite[Proposition 2.1.18(i)]{Lothaire}. More directly, it follows from Lemma 
\ref{l:ineq} below. 
\end{proof}

\begin{lemma}\label{l:ineq}
Let $s_1s_2\dots$ be a Sturmian word over $\{0,1\}$ of slope $\theta$. Then 
$\lvert \theta k-\sum_{i={n}}^{n+k-1}s_i\rvert <1$ for all $n,k\ge 1$. 
\end{lemma}

\begin{proof}
This follows from an easy computation. In the case \eqref{eq:S1}, we have 
\[\theta k-\sum_{i={n}}^{n+k-1}s_i=\theta k -(\lfloor (n+k)\theta+\rho\rfloor - \lfloor n\theta+\rho\rfloor)
=\{ (n+k)\theta+\rho\} - \{ n\theta+\rho\}\in (-1,1).
\]
The same holds in the case \eqref{eq:S2} after replacing 
$\lfloor\cdot\rfloor$ by $\lceil\cdot\rceil$ and $\{\cdot\}$ by 
$-\{-(\cdot)\}$. 
\end{proof}

\begin{lemma}
Let $\theta\in (0,1)$ be irrational. Then $0\bc_\theta$ and $1\bc_\theta$ are 
Sturmian. 
\end{lemma}

\begin{proof}
This is clear by taking $\rho=-\theta$ in \eqref{eq:S1} and \eqref{eq:S2}. 
\end{proof}

\begin{lemma}\label{l:tr0}
Let $0<r<1/2$. Then $t_r$ restricts to an order-preserving injection 
$(\{0,1\}^\N,\preclex)\to (\R,\le)$. 
\end{lemma}

\begin{proof}
Let $\bs=s_1s_2\dots$, $\bs'=s'_1s'_2\dots\in \{0,1\}^\N$ with $\bs\precnlex 
\bs'$. Then there exists $n\ge 1$ such that $s_i=s'_i$ for all $i<n$ and 
$s_n<s'_n$. Thus
\[t_r(\bs')-t_r(\bs)=\sum_{i=1}^\infty (s'_i-s_i)r^i\ge r^n-\sum_{i=n+1}^\infty r^{i}=r^n(1-\frac{r}{1-r})>0.\] 
\end{proof}

\begin{lemma}\label{l:tr}
Let $0<r<1$ and let $\bs$ and $\bs'$ be Sturmian words over $\{0,1\}$ of the 
same slope. Then $\lvert t_r(\bs')-t_r(\bs)\rvert \le r$. Moreover, if 
$\bs\precnlex \bs'$, then $t_r(\bs)<t_r(\bs')$. 
\end{lemma}

\begin{proof}
We may assume $\bs\precnlex \bs'$. Write $\bs=s_1s_2\dots$ and 
$\bs'=s'_1s'_2\dots$. Then 
\[t_r(\bs')-t_r(\bs)=\sum_{n=1}^\infty (s'_n-s_n)r^n.\]
By Lemma \ref{l:Sturm2}, the sequence $(s'_n-s_n)_{n\ge 1}$ with values in 
$\{-1,0,1\}$ alternates in sign after removing all zeroes. Moreover, the 
leading nonzero term is positive by the assumption $\bs\precnlex \bs'$. Thus 
$0<t_r(\bs')-t_r(\bs)\le r$. 
\end{proof}

The notation $t_r$ extends naturally to finite words. Let $\Sigma$ be a 
finite set of integers. For a finite word $\bu=u_1\dots u_k$ over $\Sigma$, 
we write $t_r(\bu)=\sum_{i=1}^k u_i r^i$. Then, for any finite or infinite 
word $\bv$ over $\Sigma$, we have 
\[t_r(\bu\bv)-r^kt_r(\bv)=t_r(\bu).\]
In particular, for any infinite word $\bw$ over $\Sigma$, if $T^n\bw=\bu 
T^{n+k}\bw$ and $T^m\bw=\bu' T^{m+k}\bw$, where $\bu$ and $\bu'$ are words of 
length $k$, then 
\begin{equation}\label{eq:triv} 
(t_r(T^n\bw)-t_r(T^m\bw))-r^k(t_r(T^{n+k}\bw)-t_r(T^{m+k}\bw))=t_r(\bu)-t_r(\bu').
\end{equation}

\section{The positive case}\label{s:pos}

As mentioned in the introduction, the following analogue of Theorem \ref{t:r} 
was proved by Veerman \cite[Theorem 2.1]{Veerman} in the case $r=1/2$ and by 
Bugeaud and Dubickas \cite[Theorem 2.1]{BD} in the case $r=1/b$, where $b\ge 
2$ is an integer. 

\begin{prop}\label{p:pos}
Let $\bv=v_1v_2\dots$ be an aperiodic word over a finite set of integers and 
let $r$ be a real number such that $0<r<1$. Then the sequence 
$(t_{r}(T^n\bv))_{n\ge 0}$ is contained in an interval of length $r$ if and 
only if there exist an integer $g$ and a Sturmian word $\bw=w_1w_2\dots$ over 
the alphabet $\{0,1\}$ such that $v_i=g+w_i$ for all $i\ge 1$. In this case, 
the endpoints of the interval are 
\[gr/(1-r)+t_{r}(0\bc_{\theta}),\quad gr/(1-r)+t_{r}(1\bc_{\theta}),\]
where $\theta$ denotes the slope of $\bw$, and the sequence is contained in a 
semiopen interval with these endpoints.  Furthermore, in this case, the 
sequence is contained in the open interval with these endpoints if and only 
if none of the words $T^n\bw$, $n\ge 1$ is characteristic Sturmian. 
\end{prop}

Veerman \cite[Remark following Lemma 2.4]{Veerman} already noted that (in the 
case $\bv$ over $\{0,1\}$) the first assertion of Proposition \ref{p:pos} 
holds for $0<r<1$ and the second assertion holds for $0<r\le 1/2$. We will 
see that the restriction to $r\le 1/2$ is unnecessary.

\begin{proof}
Assume that $(t_{r}(T^n\bv))_{n\ge 0}$ is contained in an interval $[A,B]$ of 
length~$r$. Let $h =\max_{n\ge 1} v_n$ and $g=\min_{n\ge 1} v_n$. Since $\bv$ 
is aperiodic, we have $h>g$. By \eqref{eq:triv}, 
\[2r>(1+r)(B-A)\ge 
t_r(h)-t_r(g)=(h-g)r.\] Thus $h-g=1$ and consequently $v_n\in \{g,g+1\}$ for 
all $n\ge 1$. Let $w_n=v_n-g\in \{0,1\}$ and $\bw=w_1w_2\dots$. Then 
$t_r(\bw)=t_r(\bv)-t_r(g^\infty)=t_r(\bv)-gr/(1-r)$. Thus, up to replacing 
$\bv$ by $\bw$ we may assume that $g=0$. Assume that $\bw$ is not Sturmian. 
By Lemma \ref{l:Sturm}, since $\bw$ is aperiodic but not Sturmian, there 
exists a word $\bu$ of length $k\ge 0$ such that $0\bu 0$ and $1\bu 1$ are 
subwords of $\bw$. Then, by \eqref{eq:triv}, we have 
\[r+r^{k+3}=(1+r^{k+2})(B-A)\ge t_r(1\bu 1)-t_r(0\bu 0)=r+r^{k+2}.\]
Contradiction. 

For the other statements of the proposition, we may again assume $g=0$. Let 
$\bw$ be a Sturmian word over $\{0,1\}$ of slope $\theta$ and let 
$W=\{t_r(T^n\bw)\mid n\ge 0\}$. Note that 
$t_r(1\bc_\theta)=t_r(0\bc_\theta)+r$. By Lemma \ref{l:tr} applied to the 
Sturmian words $T^n\bw$, $0\bc_\theta$, and $1\bc_\theta$, we have 
\begin{gather*}
t_r(T^n\bw)=t_r(1\bc_\theta)-(t_r(1\bc_\theta)-t_r(T^n\bw))\ge t_r(1\bc_\theta)-r=t_r(0\bc_\theta),\\
t_r(T^n\bw)=t_r(0\bc_\theta)+(t_r(T^n\bw)-t_r(0\bc_\theta))\le t_r(0\bc_\theta)+r=t_r(1\bc_\theta).
\end{gather*}
Thus $W$ is contained in the interval $[t_r(0\bc_\theta),t_r(1\bc_\theta)]$ 
of length $r$. By \cite[Proposition 2.3]{LZ}, $W$ is not contained in any 
shorter interval. 

By Lemma \ref{l:tr},  $t_r(0\bc_\theta)\in W$ if and only if 
$T^n\bw=0\bc_\theta$ for some $n \ge 0$, and $t_r(1\bc_\theta)\in W$ if and 
only if $T^m\bw=1\bc_\theta$ for some $m \ge 0$. If both conditions are 
satisfied, then $T^{m+1}\bw=\bc_\theta=T^{n+1}\bw$ for $m\neq n$, which 
contradicts the aperiodicity of $\bw$. Thus $W$ is contained in 
$[t_r(0\bc_\theta),t_r(1\bc_\theta))$ or 
$(t_r(0\bc_\theta),t_r(1\bc_\theta)]$. Furthermore, $W$ is not contained in 
$(t_r(0\bc_\theta),t_r(1\bc_\theta))$ if and only if $T^n\bw=\bc_\theta$ for 
some $n\ge 1$. 
\end{proof}

The following statement is taken from \cite[Theorem~1]{AG}. 

\begin{cor}[Veerman]\label{c:Veerman}
An aperiodic word $\bw$ over $\{0,1\}$ is Sturmian if and only if there 
exists an infinite word $\bu$ over $\{0,1\}$ such that $0\bu\preclex 
T^n\bw\preclex1\bu$ for all $n\ge 0$. In this case, $\bu=\bc_\theta$ and 
\[0\bu=\inf\nolimits^\lex_{n\ge 0} T^n\bw, \quad 1\bu=\sup\nolimits^\lex_{n\ge 0} T^n \bw.\] 
Here $\theta$ denotes the slope of $\bw$, and $\inf^\lex$ and $\sup^\lex$ 
denote respectively the infimum and supremum with respect to the order 
$\preclex$. 
\end{cor}

\begin{proof}
Take $0<r<1/2$. The corollary follows easily from Proposition \ref{p:pos} and 
Lemma \ref{l:tr0}. In more detail, if there exists $\bu$ such that 
$0\bu\preclex T^n\bw\preclex1\bu$, then, by Lemma \ref{l:tr0}, 
$(t_r(T^n\bw))_{n\ge 0}$ is contained in the interval $[t_r(0\bu),t_r(1\bu)]$ 
of length $r$, which implies that $\bw$ is Sturmian by Proposition 
\ref{p:pos}. Conversely, assume that $\bw$ is Sturmian of slope $\theta$. 
Then, by Proposition \ref{p:pos}, for all $n\ge 0$, $t_r(0\bc_\theta)\le 
t_r(T^n\bw)\le t_r(1\bc_\theta)$, which implies $0\bc_\theta\preclex 
T^n\bw\preclex1\bc_\theta$. If there exists an infinite word $\bv$ over 
$\{0,1\}$ such that $0\bc_\theta\precnlex \bv\preclex T^n\bw$ for all $n\ge 
0$, then, by Lemma \ref{l:tr0}, $(t_r(T^n\bw))_{n\ge 0}$ is contained in the 
interval $[t_r(\bv),t_r(1\bc_\theta)]$ of length $<r$, which contradicts 
\cite[Proposition 2.3]{LZ}. Thus $0\bc_\theta=\inf\nolimits^\lex_{n\ge 0} 
T^n\bw$ and similarly $1\bc_\theta=\sup\nolimits^\lex_{n\ge 0} T^n \bw$. 
\end{proof}

\begin{cor}[Bugeaud--Dubickas]\label{c:BD}
Let $b\ge 2$ be an integer and let $\xi$ be an irrational real number. Then 
the sequence $(\pi(\xi b^n))_{n\ge 0}$ is contained in an interval of $\R/\Z$ 
of length $1/b$ if and only if $\xi=g/(b-1)+t_{1/b}(\bw)$, where $g$ is an 
integer and $\bw$ is a Sturmian word over $\{0,1\}$. In this case, $\xi$ is 
transcendental and the endpoints of the interval are 
\[
\pi(g/(b-1)+t_{1/b}(0\bc_{\theta})),\quad \pi(g/(b-1)+t_{1/b}(1\bc_{\theta})),
\]
where $\theta$ denotes the slope of $\bw$, and the sequence is contained in a 
semiopen interval with these endpoints. Furthermore, in this case, the 
sequence is contained in the open interval with these endpoints if and only 
if none of the words $T^n\bw$, $n\ge 1$ is characteristic Sturmian.
\end{cor}

\begin{proof}
This is essentially \cite[Theorem 2.1]{BD}, except that here we do not a 
priori exclude the case where $\pi(0)$ is in the interior of the interval or 
is the upper endpoint. Assume that $(\pi(\xi b^n))_{n\ge 0}$ is contained in 
an interval $I\subseteq \R/\Z$ of length $1/b$. Choose $\eta\in \R$ such that 
$\pi(\eta)\notin I$. Then $I$ lifts to a unique interval $\tilde I\subseteq 
(\eta,\eta+1)$ of length $1/b$. For $n\ge 0$, let $y_n=\{\xi 
b^n-\eta\}+\eta\in (\eta,\eta+1)$ and $v_n=by_n-y_{n+1}\in \Z\cap 
((b-1)\eta-1,(b-1)\eta+b)$. Since $\pi(y_n)=\pi(\xi b^n)$, we have $y_n\in 
\tilde I$ for all $n\ge 0$. Let $\bv=v_0v_1\dots$. Then 
$t_{1/b}(T^n\bv)=\sum_{i=0}^\infty v_{n+i} b^{-(i+1)}=y_n$. Since $y_0$ is 
irrational, $\bv$ is aperiodic. Then, by Proposition \ref{p:pos}, there 
exists $g\in \Z$ and a Sturmian word $\bw=w_0w_1\dots $ over $\{0,1\}$ such 
that $v_n=g+w_n$ for all $n\ge 0$. Thus 
$\pi(\xi)=\pi(y_0)=\pi(g/(b-1)+t_{1/b}(\bw))$. By a theorem of Ferenczi and 
Mauduit \cite[Proposition 2]{FM}, $t_{1/b}(\bw)$ is transcendental and 
consequently the same holds for $\xi$. The other assertions of Corollary 
\ref{c:BD} follow easily from the corresponding assertions of Proposition 
\ref{p:pos} and the formula $\pi(t_{1/b}(T^n\bv))=\pi(b^nt_{1/b}(\bv))$. 
\end{proof}

\section{Proofs of the main results}\label{s:proof}

We need a slight generalization of Sturmian words.

\begin{notation}
For $l\ge 0$, we let $\cS_l$ denote the collection of infinite words over 
$\{0,1\}$ that are of the form $0^l\bs$ or $1^l\bs$, where $\bs$ is a 
Sturmian word. We write $\cS=\bigcup_{l\ge 0}\cS_l$.
\end{notation}

The following characterization of the class $\cS$ parallels Lemma 
\ref{l:Sturm}. 

\begin{lemma}\label{l:S}
Let $\bw$ be an aperiodic word over $\{0,1\}$. Then $\bw\in \cS$ if and only 
if there does not exist a finite word $\bu$ such that $01\bu1$ and $10\bu 0$ 
are subwords of~$\bw$. 
\end{lemma} 

\begin{proof}
The ``only if'' part.  We have $\bw=0^l\bs$ or $\bw=1^l\bs$ with $l\ge 0$ and 
$\bs$ Sturmian. If $01\bu 1$ and $10\bu 0$ are subwords of $\bw$, then 
$0\bu0$ and $1\bu 1$ are subwords of $\bs$, which contradicts Lemma 
\ref{l:Sturm}. 

The ``if'' part. We have $\bw=0^l\bs$ with $l\ge 1$ and $\bs=1\dots$ or 
$\bw=1^l\bs$ with $l\ge 1$ and $\bs=0\dots$. Assume that $\bs$ is not 
Sturmian. By Lemma \ref{l:Sturm}, there exists a finite word $\bu$ such that 
$0\bu0$ and $1\bu 1$ are subwords of~$\bs$. Assume that $01\bu1$ is a subword 
of $\bw$. Then, by assumption, $10\bu0$ is not a subword of $\bw$. It follows 
that $00\bu 0$ is a subword of $\bs$. Let $k\ge 0$ be the number of leading 
$0$s in $\bu$. Then $010^k1$ is a subword of $01\bu 1$ and thus a subword of 
$\bw$. Moreover, $0^{k+2}$ is a subword of $00\bu$ and thus a subword of 
$\bs$, and consequently $10^{k+2}=100^k0$ is a subword of $\bw$. This 
contradicts the assumption. Therefore, $01\bu1$ is not a subword of $\bw$ and 
consequently $11\bu 1$ is a subword of $\bs$. Similarly, $00\bu0$ is a 
subword of $\bs$. It follows that $011$ and $100$ are subwords of $\bw$. 
Contradiction. 
\end{proof}

The order-preserving property in Lemma \ref{l:tr} extends to $\cS$. 

\begin{lemma}
Let $0<r<1$. Let $\bs,\bs'\in \cS$ be of the same slope and such that 
$\bs\precnlex \bs'$. Then $t_r(\bs)<t_r(\bs')$. 
\end{lemma}

\begin{proof}
Assume $\bs\in \cS_l$ and $\bs'\in \cS_{l'}$. Let $n=\max(l,l')$. Write 
$\bs=\bu T^n\bs$ and $\bs'=\bu'T^n\bs'$, where $\bu$ and $\bu'$ are words of 
length $n$. Then $\{\bu,\bu'\}\cap \{0^n,1^n\}$ is nonempty. Moreover, 
$T^n\bs$ and $T^n\bs'$ are Sturmian. We have 
\[t_r(\bs')-t_r(\bs)=(t_r(\bu')-t_r(\bu))+r^n(t_r(T^n\bs')-t_r(T^n\bs)).\]  
There are two cases for $\bs\precnlex \bs'$. 

Case $\bu=\bu'$ and $T^n\bs\precnlex T^n\bs'$. In this case, 
$t_r(\bs')-t_r(\bs)=r^n(t_r(T^n\bs')-t_r(T^n\bs))>0$ by Lemma \ref{l:tr}. 

Case $\bu\precnlex \bu'$. Since either $\bu=0^n$ or $\bu'=1^n$, we have 
$t_r(\bu')-t_r(\bu)\ge r^n$. Moreover, $\lvert 
t_r(T^n\bs')-t_r(T^n\bs)\rvert\le r$ by Lemma \ref{l:tr}. Thus 
$t_r(\bs')-t_r(\bs)\ge r^n-r^{n+1}>0$.
\end{proof}

The classes $\cS$ and $\cD$ are related as follows. 

\begin{lemma}\label{l:D}
We have $\cD=D(\cS)\coprod T(D(\cS))$.
\end{lemma}

\begin{proof}
By definition $D(\cS)\subseteq \cD$ and $T(D(\cS_l))\subseteq \cD$ for $l\ge 
1$. For $\bs=s_1\dots$ Sturmian, $T(D(\bs))=s_1D(T(\bs))\in \cD$. Thus 
$D(\cS)\cup T(D(\cS))\subseteq \cD$. 

For $\bw\in \cD$, we have $\bw=a^lD\bs$ for $a\in \{0,1\}$, $l\ge 0$ and 
$\bs$ Sturmian. If $l=2k$ is even, then $\bw=D(a^k\bs)$. If $l=2k-1$ is odd, 
then $\bw=T(D(a^k\bs))$. Thus $\cD=D(\cS)\cup T(D(\cS))$.

Finally, for $\Sigma=\{0,1\}$, $D(\Sigma^\N)\cap 
T(D(\Sigma^\N))=\{0^\infty,1^\infty\}$. Thus $D(\cS)\cap T(D(\cS))$ is empty. 
\end{proof}

Since $\cS$ is stable under $T$, it follows from Lemma \ref{l:D} (or the 
definition) that $\cD$ is also stable under $T$.  

Next we analyze the alternate order on $\cD$. It is convenient to do this for 
a more general class of words. 

\begin{notation}\label{n:P}
Let $\Sigma=\{0,1\}$, $\cP=D(\Sigma^\N)\cup T(D(\Sigma^\N))$. We put 
\begin{align*}
\cP_{1^\infty]}&=\{\bw\in \cP\mid \bw\precalt 1^\infty\},\\ 
\cP_{[1^\infty,0^\infty]}&=\{\bw\in \cP\mid 1^\infty \precalt\bw\precalt 0^\infty\},\\ 
\cP_{[0^\infty}&=\{\bw\in \cP\mid 0^\infty\precalt \bw\}.
\end{align*}
\end{notation} 

\begin{lemma}\label{l:P}\leavevmode
\begin{enumerate}
\item The map $(\Sigma^\N,\preclex)\to (\cP_{1^\infty]},\precalt)$ carrying 
    $\bu$ to $1D\bu$ is an order-preserving bijection. 
\item The map $(\Sigma^\N,\preclex)\to 
    (\cP_{[1^\infty,0^\infty]},\precalt)$ carrying $\bu$ to $D\bu$ is an 
    order-reversing bijection. 
\item The map $(\Sigma^\N,\preclex)\to (\cP_{[0^\infty},\precalt)$ carrying 
    $\bu$ to $0D\bu$ is an order-preserving bijection. 
\end{enumerate}
\end{lemma}

\begin{proof}
The bijectivity follows from the fact that for any $\bw\in \Sigma^\N$, 
$\bw\precnalt 1^\infty$ if and only if the number of leading $1$s in $\bw$ is 
odd and $0^\infty\precnalt \bw$ if and only if the number of leading $0$s in 
$\bw$ is odd. The preservation and reversal of the order follow from the 
definition. 
\end{proof}

\begin{lemma}\label{l:t-r}
Let $0<r<1/2$. Then $t_{-r}$ restricts to an order-preserving injection 
$(\Sigma^\N,\precalt)\to (\R,\le)$. 
\end{lemma}

\begin{proof}
Let $\bs=s_1s_2\dots$, $\bs'=s'_1s'_2\dots\in \{0,1\}^\N$ with $\bs\precnalt 
\bs'$. Then there exists $n\ge 1$ such that $s_i=s'_i$ for all $i<n$ and 
$(-1)^n(s'_n-s_n)=1$. Thus
\[t_{-r}(\bs')-t_{-r}(\bs)=\sum_{i=1}^\infty (s'_i-s_i)(-r)^i\ge r^n-\sum_{i=n+1}^\infty r^{i}=r^n(1-\frac{r}{1-r})>0.\] 
\end{proof}

\begin{lemma}\label{l:P2}
Let $0<r<1$ and let $\bw \in \cP$.
\begin{enumerate}
\item If $\bw\precalt 1^\infty$, then $t_{-r}(\bw)\le t_{-r}(1^\infty)$ and 
    equality holds if and only if $\bw=1^\infty$. 
\item If $1^\infty\precalt\bw\precalt 0^\infty$, then $t_{-r}(1^\infty)\le 
    t_{-r}(\bw)\le t_{-r}(0^\infty)$ and the first (resp.\ second) equality 
    holds if and only if $\bw=1^\infty$ (resp.\ $\bw=0^\infty$). 
\item If $0^\infty\precalt \bw$, then $t_{-r}(0^\infty)\le t_{-r}(\bw)$ and 
    equality holds if and only if $\bw=0^\infty$. 
\end{enumerate} 
\end{lemma}

\begin{proof}
The inequalities follow from Lemma \ref{l:P} and the formulas 
\begin{equation}\label{eq:P2}
t_{-r}(D\bu)=-(r^{-1}-1)t_{r^2}(\bu),\quad t_{-r}(aD\bu)=-ar+(1-r)t_{r^2}(\bu),
\end{equation}
and $t_{r^2}(\bu)\in [t_{r^2}(0^\infty),t_{r^2}(1^\infty)]$, where $a\in 
\Sigma$ and $\bu\in \Sigma^\N$. The equality conditions follow from the fact 
that, for $a\in \{0,1\}$, $t_{r^2}(\bu)=t_{r^2}(a^\infty)$ if and only if 
$\bu=a^\infty$. 
\end{proof} 

\begin{lemma}\label{l:Dbounds}
Let $\bw\in \cD$ and let $\theta$ be the slope of $\bw$. Then 
$100D\bc_{\theta}\precalt \bw\precalt 011D\bc_{\theta}$.
\end{lemma}

\begin{proof}
Assume $\bw\precalt 100D\bc_{\theta}$. Then $\bw=10\dots$. Thus $\bw=1D\bs$, 
where $\bs$ is Sturmian of slope $\theta$. Since $\bw\precalt 
1D(0\bc_{\theta})$, we have $\bs\preclex 0\bc_\theta$ by Lemma \ref{l:P}. By 
Corollary \ref{c:Veerman}, we must have $\bs=0\bc_\theta$ and consequently 
$\bw=100D\bc_\theta$. Similarly,   $011D\bc_{\theta}\precalt \bw$ implies 
$\bw=011D\bc_\theta$. 
\end{proof}

\begin{lemma}\label{l:Dtr}
Let $0<r<1$. Let $\bw,\bw'\in \cD$ be of the same slope and such that 
$\bw\precnalt \bw'$. Then $t_{-r}(\bw)<t_{-r}(\bw')$. 
\end{lemma}

\begin{proof}
By Lemma \ref{l:P2}, we may assume that one of the three intervals in 
Notation \ref{n:P} contains both $\bw$ and $\bw'$. 

Case $\bw,\bw'\in \cP_{1^\infty]}$ or $\bw,\bw'\in \cP_{[0^\infty}$. By Lemma 
\ref{l:P}, $\bw=aD\bu$, $\bw'=aD\bu'$, where $a\in \{0,1\}$ and $\bu\precnlex 
\bu'$. Then $t_{r^2}(\bu)<t_{r^2}(\bu')$ by Lemma \ref{l:tr} and consequently 
$t_{-r}(\bw)<t_{-r}(\bw')$ by \eqref{eq:P2}. 

Case $\bw,\bw'\in \cP_{[1^\infty,0^\infty]}$. By Lemma \ref{l:P}, $\bw=D\bu$, 
$\bw'=D\bu'$, where $\bu'\precnlex \bu$. Then $t_{r^2}(\bu')<t_{r^2}(\bu)$ by 
Lemma \ref{l:tr} and consequently $t_{-r}(\bw)<t_{-r}(\bw')$ by 
\eqref{eq:P2}. 
\end{proof}

\begin{proof}[Proof of Theorem \ref{t:r}]
The beginning of the proof of the ``only if'' part of the first assertion is 
similar to that of \cite[Theorem 1.2]{LZ}. Assume that 
$(t_{-r}(T^n\bv))_{n\ge 0}$ is contained in an interval $[A,B]$ of length 
$r+r^2-r^3$. Let $\mu =\max_{n\ge 1} v_n$ and $\nu=\min_{n\ge 1} v_n$. Since 
$\bv$ is aperiodic, we have $\mu>\nu$. Moreover, there exists a subword 
$\mu\mu'$ of $\bv$ with $\mu'<\mu$. Similarly, there exists a subword 
$\nu\nu'$ of $\bv$ with $\nu'>\nu$. If $\mu-\nu \ge 2$, then, by 
\eqref{eq:triv}, 
\begin{multline*}
2r-r(1-r)(1-r^3)=(1+r^2)(B-A)\ge t_{-r}(\nu\nu')-t_{-r}(\mu\mu')\\
= -(\nu-\mu)r+(\nu'-\mu')r^2\ge (\mu-\nu)r-(\mu-\nu-2)r^2\ge 2r,
\end{multline*}
which is a contradiction. Thus $\mu-\nu=1$ and consequently $v_n\in 
\{\nu,\nu+1\}$ for all $n\ge 1$. Let $g=-\nu$, $w_n=v_n+g\in \{0,1\}$ and 
$\bw=w_1w_2\dots $. Then 
$t_{-r}(\bw)=t_{-r}(\bv)+t_{-r}(g^\infty)=t_{-r}(\bv)-gr/(1+r)$. Up to 
replacing $\bv$ by $\bw$, we may assume $g=0$. By the above $01$ and $10$ are 
subwords of $\bw$. 

Note that $\bw$ has a subword of the form $10a$, where $a\in \{0,1\}$. If 
$010$ is a subword of $\bw$, then, by \eqref{eq:triv}, we have 
\[r+r^2-r^3(1-r)(1-r^2)=(1+r^3)(B-A)\ge t_{-r}(010)-t_{-r}(10a)=r+r^2+ar^3\ge r+r^2,\]
which is a contradiction. Thus $010$ is not a subword of $\bw$. Similarly, 
$101$ is not a subword of $\bw$. 

Next note that $0111$ and $1000$ are not both subwords of $\bw$. Indeed, 
otherwise, by \eqref{eq:triv}, we would have
\[r+r^2-r^3+r^4-r^4(1-r)(1-r^2)=(1+r^4)(B-A)\ge t_{-r}(0111)-t_{-r}(1000)=r+r^2-r^3+r^4.\] 
Up to replacing $w_n$ by $1-w_n$ for all $n\ge 1$, we may assume that $0111$ 
is not a subword of $\bw$. Then all blocks of 0s in $\bw$ are separated by 
$11$. Thus $\bw=1^y0^{z_0}110^{z_1}110^{z_2}\dots$, where $y\ge 0$, $z_0\ge 
1$, and $z_n\ge 2$ for all $n\ge 1$.

Assume that $z_n$ is odd for some $n\ge 1$. Then $0110^{z_n-1}$ and 
$10^{z_n}1$ are subwords of $\bw$. Thus, by \eqref{eq:triv}, we have 
\begin{multline}\label{eq:odd}
r+r^2-r^3+r^{z_n+2}-r^{z_n+2}(1-r)(1-r^2)=(1+r^{z_n+2})(B-A)\\
\ge t_{-r}(0110^{z_n-1})-t_{-r}(10^{z_n}1)=r+r^2-r^3+r^{z_n+2}.
\end{multline}
Contradiction.  It follows that $z_n$ is even for all $n\ge 1$.

The proof now diverges from the proof of \cite[Theorem 1.2]{LZ}.  We claim 
that $\bw\in \cP$. This holds if $y=0$ or $z_0$ is even. Assume that $y>0$ 
and $z_0$ is odd. If $z_0\le z_m+1$ for some $m\ge 1$, then $0110^{z_0-1}$ 
and $10^{z_0}1$ are subwords of $\bw$ and thus, by \eqref{eq:triv}, the 
inequality \eqref{eq:odd} holds for $n=0$, which is a contradiction. It 
follows that $z_0>z_n+1$ and consequently $z_0\ge z_n+3$ for all $n\ge 1$. 
Then $0110^{z_n}1$ and $10^{z_n+3}$ are subwords of $\bw$. Thus, by 
\eqref{eq:triv}, we have 
\begin{multline*} 
r+r^2-r^3+r^{z_n+4}-r^{z_n+4}(1-r)(1-r^2)=(1+r^{z_n+4})(B-A)\\
\ge t_{-r}(0110^{z_n}1)-t_{-r}(10^{z_n+3})=r+r^2-r^3+r^{z_n+4}.
\end{multline*}
Contradiction. 

This finishes the proof of the claim $\bw\in \cP$. It follows that there 
exists an infinite word $\bs$ such that $\bw=D\bs$ or $\bw=T(D\bs)$. Assume 
that $\bs\notin \cS$. Then, by Lemma \ref{l:S}, there exists a finite word 
$\bu$ such that $01\bu1$ and $10\bu 0$ are subwords of~$\bs$. It follows that 
$011(D\bu) 1$ and $100(D\bu)0$ are subwords of~$\bw$. Let $k$ be the length 
of $\bu$.  Then, by \eqref{eq:triv}, we have 
\begin{multline*} 
r+r^2-r^3+r^{2k+4}-r^{2k+4}(1-r)(1-r^2)=(1+r^{2k+4})(B-A)\\\ge t_{-r}(011(D\bu) 1)- t_{-r}(100(D\bu)0)= r+r^2-r^3+r^{2k+4}.
\end{multline*}
Contradiction. Therefore, $\bs\in \cS$ and consequently $\bw\in \cD$.

For the other statements of the theorem, we may again assume $g=0$. Let 
$\bw\in \cD$ and let $\theta$ be the slope of $\bw$. Let 
$W=\{t_{-r}(T^n\bw)\mid n\ge 0\}$. By Lemmas \ref{l:Dbounds} and \ref{l:Dtr}, 
since $T^n\bw\in \cD$ for all $n\ge 0$, $W$ is contained in the interval 
$[t_{-r}(100D\bc_\theta),t_{-r}(011D\bc_\theta)]$ of length $r+r^2-r^3$. By 
\cite[Theorem 1.2]{LZ}, $W$ is not contained in any shorter interval. By 
Lemma \ref{l:Dtr},  $t_{-r}(100D\bc_\theta)\in W$ if and only if 
$T^n\bw=100D\bc_\theta$ for some $n \ge 0$, and $t_{-r}(011D\bc_\theta)\in W$ 
if and only if $T^m\bw=011D\bc_\theta$ for some $m \ge 0$. If both conditions 
hold, then $T^{m+3}\bw=D\bc_\theta=T^{n+3}\bw$ for $m\neq n$, which 
contradicts the aperiodicity of $\bw$. Thus $W$ is contained in 
$[t_{-r}(100D\bc_\theta),t_{-r}(011D\bc_\theta))$ or 
$(t_{-r}(100D\bc_\theta),t_{-r}(011D\bc_\theta)]$. Furthermore, $W$ is 
contained in $(t_{-r}(100D\bc_\theta),t_{-r}(011D\bc_\theta))$ if and only if 
$\bw\notin \cC$. 
\end{proof}

To show the transcendence in Theorem \ref{t:main}, we use the theorem of 
Ferenczi and Mauduit \cite[Proposition~2]{FM} as follows. 

\begin{lemma}\label{l:trans}
Let $b\ge 2$ be an integer and let $\bw\in \cD$. Then $t_{-1/b}(\bw)$ is 
transcendental. 
\end{lemma} 

\begin{proof}
We have $\bw=a^lD\bs$, where $a\in \{0,1\}$, $l\ge 0$, and $\bs$ is Sturmian. 
Then $t_{-1/b}(\bw)=t_{-1/b}(a^l)+(-b)^{-l}(-b+1)t_{1/b^2}(\bs)$. By 
\cite[Proposition~2]{FM}, $t_{1/b^2}(\bs)$ is transcendental. Thus the same 
holds for $t_{-1/b}(\bw)$. 
\end{proof}

\begin{proof}[Proof of Theorem \ref{t:main}]
We apply Theorem \ref{t:r} to $r=1/b$ as follows. Assume that 
$(\pi(\xi(-b)^n))_{n\ge 0}$ is contained in an interval $I\subseteq \R/\Z$ of 
length $b^{-1}+b^{-2}-b^{-3}<1$. Choose $\eta\in \R$ such that 
$\pi(\eta)\notin I$. Then $I$ lifts to a unique interval $\tilde I\subseteq 
(\eta,\eta+1)$ of length $b^{-1}+b^{-2}-b^{-3}$. For $n\ge 0$, let $y_n=\{\xi 
(-b)^n-\eta\}+\eta\in (\eta,\eta+1)$ and $v_n=-by_n-y_{n+1}\in \Z\cap 
(-(b+1)(\eta+1),-(b+1)\eta)$. Since $\pi(y_n)=\pi(\xi (-b)^n)$, we have 
$y_n\in \tilde I$ for all $n\ge 0$. Let $\bv=v_0v_1\dots$. Then 
$t_{-1/b}(T^n\bv)=\sum_{i=0}^\infty v_{n+i} (-b)^{-(i+1)}=y_n$. Since $y_0$ 
is irrational, $\bv$ is aperiodic. Then, by Theorem \ref{t:r}, there exist 
$g\in \Z$ and $\bw=w_0w_1\dotso\in \cD$ such that $v_i=-g+w_i$ for all $i\ge 
0$. Thus $\pi(\xi)=\pi(y_0)=\pi(g/(b+1)+t_{-1/b}(\bw))$. By Lemma 
\ref{l:trans}, it follows that $\xi$ is transcendental. The other assertions 
of Theorem \ref{t:main} follow easily from the corresponding assertions of 
Theorem \ref{t:r} and the formula  
$\pi(t_{-1/b}(T^n\bv))=\pi((-b)^nt_{-1/b}(\bv))$. 
\end{proof}

\begin{proof}[Proof of Corollary \ref{c:alt}]
While this can be deduced from Theorem \ref{t:main}, we deduce it more 
directly from Theorem \ref{t:r}, in parallel with the proof of Corollary 
\ref{c:Veerman}. Take $0<r<1/2$. If there exists $\bu$ such that 
$100\bu\precalt T^n\bw\precalt011\bu$, then, by Lemma \ref{l:t-r}, 
$(t_{-r}(T^n\bw))_{n\ge 0}$ is contained in the interval 
$[t_{-r}(100\bu),t_{-r}(011\bu)]$ of length $r+r^2-r^3$, which implies 
$\bw\in \cD$ by Theorem \ref{t:r}. Conversely, assume that $\bw\in \cD$. Let 
$\theta$ be the slope of $\bw$. We have already proven  
$100D\bc_\theta\precalt T^n\bw\precalt011D\bc_\theta$ in Lemma 
\ref{l:Dbounds}. If there exists an infinite word $\bv$ over $\{0,1\}$ such 
that $100D\bc_\theta\precnalt \bv\precalt T^n\bw$ for all $n\ge 0$, then, by 
Lemma \ref{l:t-r}, $(t_{-r}(T^n\bw))_{n\ge 0}$ is contained in the interval 
$[t_{-r}(\bv),t_{-r}(011D\bc_\theta)]$ of length $<r+r^2-r^3$, which 
contradicts \cite[Theorem 1.2]{LZ}. Thus 
$100D\bc_\theta=\inf\nolimits^\alt_{n\ge 0} T^n\bw$ and similarly 
$011D\bc_\theta=\sup\nolimits^\alt_{n\ge 0} T^n \bw$. 
\end{proof}

\begin{remark}
The interval with endpoints \eqref{eq:ex} in Theorem \ref{t:main} is a 
neighborhood of $\pi([(g-1)/(b+1)-(b^{-2}-b^{-3}),g/(b+1)+(b^{-2}-b^{-3})])$, 
where $g\in \Z$. In particular, the union of all the open intervals of 
$\R/\Z$ of length $b^{-1}+b^{-2}-b^{-3}$ containing $(\pi(\xi(-b)^n))_{n\ge 
0}$ for some irrational real number $\xi$ is the entire $\R/\Z$. This stands 
in contrast to the case of positive common ratio (\cite[Theorem 2.1]{BD} or 
Corollary \ref{c:BD}), where the closure of each interval of length $1/b$ 
containing $(\pi(\xi b^n))_{n\ge 0}$ for some irrational real number $\xi$ is 
contained in $\pi((g/(b-1),(g+1)/(b-1)))$ for some integer~$g$. 

Dubickas (\cite[Theorem~3]{Dubickas}, \cite[Theorem~2]{Dubickas-neg}) showed 
that for any $\epsilon>0$ and any irrational real number $\xi$, 
$(\pi(\xi(-b)^n))_{n\ge 0}$ is not contained in any of the following four 
intervals: 
\[\pi([-A+\epsilon,A-\epsilon]),\quad \pi([\epsilon,B-\epsilon]),\quad \pi([1/2-A'+\epsilon,1/2+A'-\epsilon]), \quad \pi([1-B+\epsilon,1-\epsilon]),\]
where $B=1-\prod_{k=1}^\infty (1-b^{-(2^k+(-1)^{k-1})/3})$, 
$A=(1-(1-b^{-1})P)/2$, $P=\prod_{k=0}^\infty (1-b^{-2^k})$, $A'=A$ if $b$ is 
odd and $A'=(1-P)/2$ if $b$ is even. Theorem \ref{t:main} implies that the 
interval with endpoints \eqref{eq:ex} is not contained in any of the four 
intervals. 
\end{remark}

\end{document}